\documentclass[reqno,centertags,12pt,draft]{amsart}

\usepackage{amssymb,amsmath,amsthm,mathtools}
\usepackage[T1]{fontenc}

\theoremstyle{plain} \newtheorem{theorem}{Theorem}[section]
\newtheorem{lemma}[theorem]{Lemma}

\newtheorem{prop}[theorem]{Proposition}
\theoremstyle{definition} \newtheorem{definition}[theorem]{Definition}
\theoremstyle{remark}

\newcommand{\R}{{\mathbb R}}
\newcommand{\C}{{\mathbb C}}
\newcommand{\Z}{{\mathbb Z}}
\newcommand{\N}{{\mathbb N}}
\newcommand{\T}{{\mathbb T}}
\newcommand{\cC}{{\mathcal C}}
\newcommand{\cF}{{\mathcal F}}
\newcommand{\cO}{{\mathcal O}}
\newcommand{\cI}{{\mathcal I}}
\newcommand{\fD}{{\mathfrak D}}

\newcommand{\ls}{\lesssim}
\newcommand{\la}{\langle}
\newcommand{\ra}{\rangle}
\newcommand{\eps}{\varepsilon}
\newcommand{\p}{\partial}
\newcommand{\wto}{\rightharpoonup}
\newcommand{\Emax}{{E_{\mathrm{max}}}}
\newcommand{\balpha}{{\boldsymbol\alpha}}
\newcommand{\cf}{{\boldsymbol{1}}}
\DeclareMathOperator{\supp}{supp}
\DeclareMathOperator{\diag}{diag}

\numberwithin{equation}{section}

\begin{document}
\title[Global well-posedness on rectangular tori]{Global Well-posedness of the Energy-Critical Defocusing NLS on rectangular tori in three dimensions}
\author[N.~Strunk]{Nils~Strunk}

\address{Universit\"at Bielefeld, Fakult\"at f\"ur Mathematik, Postfach 100131, 33501 Bielefeld, Germany}
\email{strunk@math.uni-bielefeld.de}

\begin{abstract}
  The energy-critical defocusing nonlinear Schr\"odinger equation on $3$-dimensional rectangular tori
  is considered. We prove that the global well-posedness result for the standard torus of
  Ionescu and Pausader extends to this class of manifolds, namely, for any initial data in $H^1$ the
  solution exists globally in time.
\end{abstract}

\subjclass[2010]{Primary 35Q55; Secondary 35B33}
\keywords{}

\maketitle

\section{Introduction}\label{sect:intro}
\noindent
Starting with Bourgain's work in 1993 \cite{B93a}, the nonlinear Schr\"odinger equation on rectangular tori
have been studied quite intensively. Several authors contributed to today's knowledge about the equation on
this domain, e.g.\ \cite{B93a,HTT11,IP12b,HTT13,B13,W13a,GOW13,S14,D14,BD14,KV14} to list just a few.
We will briefly summarize some important known results of the nonlinear Schr\"odinger equation on $3$-dimensional
tori:
On arbitrary rectangular tori, the author of the present paper proved a trilinear Strichartz estimate that is
sufficient for achieving local well-posedness and global well-posedness for small data in the energy-critical
space in the focusing and defocusing case \cite[Proposition~3.3 and Theorem~1.1]{S14}.
Global well-posedness for the defocusing NLS with arbitrarily large initial data in the energy space $H^1$ have
been proved by Ionescu and Pausader on the standard torus $\T^3=(\R/2\pi\Z)^3$ \cite[Theorem~1.1]{IP12b}.
The strength of their argument is that large parts of the proof hold true even for a general smooth compact
Riemannian $3$-manifold such as for the domain $\mathbb S^3$ in \cite{PTW14}.
Thomann \cite[Theorem~1.4]{T08} proved (for a general analytic manifold) that the Cauchy problem is ill-posed in
$H^1$ for superquintic nonlinearities.
In the present paper, we show that Ionescu and Pausader's arguments in \cite{IP12b} can be extended to any
rectangular torus, where the periods may have irrational ratio. Hence, we obtain global well-posedness
for any initial data in $H^1$. In view of Thomann's result, this will complete the local and global
well-posedness in $H^1$ on this class of manifolds.

\medskip

In the following, we consider a general rectangular torus, i.e.\ given any
$\balpha=(\alpha_1,\alpha_2,\alpha_3)\in(0,\infty)^3$, we define
\[
  \T^3_{\balpha} \coloneqq \R^3/(2\pi\alpha_1\Z \times 2\pi\alpha_2\Z \times 2\pi\alpha_3\Z).
\]
For notational convenience, we will use the standard torus $\T^3\coloneqq\T^3_{(1,1,1)}$ as base space:
By a change of spatial variables, we rewrite the defocusing nonlinear Schr\"odinger equation 
\[
   \left\{
     \begin{array}{rcll}
       i\partial_t v+\sum_{j=1}^3 \frac{\partial^2}{\partial x_j^2}v &=& v|v|^4 \\
       v|_{t=0}                &=& \widetilde \phi \in H^1(\T^3_\balpha),
     \end{array}
   \right.
   \quad (t,x)\in(-T,T)\times\T^3_\balpha
\]
  as
\begin{equation} 
  \label{eq:nls}
   \left\{
     \begin{array}{rcll}
       i\partial_t u+\Delta u &=& u|u|^4 \\
       u|_{t=0}                &=& \phi \in H^1(\T^3),
     \end{array}
   \right.
   \quad (t,x)\in(-T,T)\times\T^3,
\end{equation}
where $\phi(x_1,x_2,x_3) \coloneqq \widetilde \phi(\alpha_1x_1,\alpha_2x_2,\alpha_3x_3)$, $(x_1,x_2,x_3)\in\T^3$.
Here, $\Delta$ is defined via
\[
   \cF(\Delta f)(n) \coloneqq -Q(n) \cF(f)(n),\quad Q(n) \coloneqq \theta_1 n_1^2 + \theta_2 n_2^2 + \theta_3 n_3^2,
\]
for $n=(n_1,n_2,n_3)\in\Z^3$ and $\theta_j\coloneqq \alpha_j^{-2}$.
Using this notation, the free solution to \eqref{eq:nls} is given by
\[
   (e^{it\Delta}\phi)(x) = \sum_{n\in\Z^3} \widehat \phi(n) e^{i(n\cdot x - Q(n)t)}.
\]
We remark that by a change of variable in time, we may assume $\theta_1=1$ without any loss of generality.

For $s\in\R$ we define the Sobolev space $H^s(\T^3)\coloneqq(1-\Delta)^{-\frac{s}{2}}L^2(\T^3)$ endowed with the norm
\[
   \|f\|^2_{H^s(\T^3)} \coloneqq \sum_{\xi\in\Z^3} \langle \xi\rangle^{2s} |\widehat f(\xi)|^2_{L^2(\T^3)},
   \quad\text{where }\langle x\rangle = (1+|x|^2)^{\frac12}.
\]

The energy and the mass,
\begin{equation}\label{eq:conserve}
   \begin{split}
     E\bigl(u(t)\bigr)&= \frac12 \int_{\T^3} |\widetilde\nabla u(t,x)|^2\,dx
                       + \frac{1}{6} \int_{\T^3} |u(t,x)|^6\,dx,\\
     M\bigl(u(t)\bigr)&= \frac12 \int_{\T^3} |u(t,x)|^2\,dx,
   \end{split}
\end{equation}
are conserved in time, whenever $u\colon (-T,T)\times \T^d\to\C$ is a strong solution of \eqref{eq:nls}.
Here $\widetilde \nabla \coloneqq (\alpha_1^{-1}\partial_{x_1},\alpha_2^{-1}\partial_{x_2},\alpha_3^{-1}\partial_{x_3})$.
The scaling-critical space to the corresponding problem in $\R^3$ is $H^1(\R^3)$ and that is why we call
\eqref{eq:nls} $H^1$-critical or, due to the same scaling, energy-critical.

\medskip

The precise formulation of our well-posedness results is a follows:

\begin{theorem}\label{thm:gwp}
  If $\phi\in H^1(\T^3)$, then there exists a unique global solution $u\in X^1(\R)$ of the initial-value problem
  \eqref{eq:nls}. Moreover, the mapping $\phi\mapsto u$ extends to a continuous mapping from $H^1(\T^3)$ to
  $X^1([-T,T])$ for any $T\in[0,\infty)$, and the quantities $M(u)$ and $E(u)$ defined in \eqref{eq:conserve} are
  conserved along the flow.
\end{theorem}



\medskip

In order to point out the similarities and differences to \cite{IP12b}, we organize the paper in the same way.
We will fix some notation and state the known Strichartz estimates in Section~\ref{sect:notation}.
In Section~\ref{sect:lpw_stability}, we state the large-data local well-posedness and stability results that,
thanks to \cite[Theorem~1.1]{KV14} and \cite[Proposition~4.1]{S14}, stays the same as in \cite[Section~3]{IP12b}.
Section~\ref{sect:eucl_prof} is destined to study free and nonlinear solutions to \eqref{eq:nls} with
initial-data that concentrate to a point in space and time. To this purpose, we prove a version of
the extinction lemma that applies to rectangular tori other than $\T^3$ as well, which replaces
\cite[Lemma~4.3]{IP12b}.
Statements about the profile decomposition that have already been proved and used in \cite{IP12a,IP12b} are
summarized in Section~\ref{sect:prof_decomp}.
Theorem~\ref{thm:gwp} may be proved by induction on energy along the lines of \cite[Section~6--7]{IP12b},
except of \cite[Lemma~7.1]{IP12b}. A replacement for this lemma is given in Section~\ref{sect:proof_thm}.

We want to point out that the paper is not self-contained but relies heavily on \cite{IPS12,IP12a,IP12b}.

\subsection*{Acknowledgments}
The author acknowledges support from the German Research Foundation,
Collaborative Research Center 701. The author also would like to thank Beno\^{i}t Pausader and Alexandru Ionescu
for helpful discussions.

\section{Preliminaries}\label{sect:notation}
\noindent
Notation and some known results are collected in this section.

We write $A\ls B$ if there is a constant $C>0$ such that $A\leq CB$. If we want to emphasize the dependence of
the constant, then we write $A\ls_s B$ for $A\leq C(s) B$, where the constant $C(s)$ depends on $s$.

For a vector $p\in\N^n$, we denote by
$\fD_{p_1,\ldots,p_n}(a_1,\ldots,a_n)$ a $|p|$-linear expression which is a product of $p_1$ terms
that are either equal to $a_1$ or its complex conjugate $\overline{a_1}$ and similarly for $p_j$,
$a_j$, $2\leq j \leq n$. 

We will use the following convention for the Fourier transform on $\T^3$
\[
  (\cF f)(\xi) = \frac{1}{(2\pi)^{3/2}} \int_{\T^3} f(x)e^{-ix\cdot\xi} \, dx,\quad \xi\in\Z^3
\]
so that we have the Fourier inversion formula
\[
   f(x) = \frac{1}{(2\pi)^{3/2}} \sum_{\xi\in\Z^3} (\cF f)(\xi)e^{ix\cdot\xi},\quad x\in\T^3.
\]

We fix a smooth, non-negative, even function $\eta^1\colon \R\to[0,1]$ with $\eta^1(y)=1$ for $|y|\leq1$
and $\supp \eta^1\subseteq(-2,2)$. Then, let $\eta^3\colon \R^3\to[0,1]$ be defined via
$\eta^3(x)\coloneqq\eta^1(x_1)\eta^1(x_2)\eta^1(x_3)$. For a dyadic number $N\geq 1$, we set
\[
   \eta_N^3(\xi) \coloneqq \eta^3\biggl(\frac{|\xi|}{N}\biggr) - \eta^3\biggl(\frac{2|\xi|}{N}\biggr)
   \quad \text{for }N\geq 2,\quad \eta^3_1(\xi) \coloneqq \eta^3(|\xi|).
\]
We also define the frequency localization operators $P_N\colon L^2(\T^d)\to L^2(\T^d)$ as the Fourier multiplier
with symbol $\eta^3_N$. Moreover, we define $P_{\leq N}\coloneqq\sum_{1\leq M\leq N} P_M$.
More generally, given a set $\mathcal S\subseteq\Z^3$, we define $P_{\mathcal S}$ to be the Fourier multiplier
operator with symbol $\cf_{\mathcal S}$, where $\cf_{\mathcal S}$ denotes the characteristic function of $\mathcal S$.

\subsection{Function spaces}\label{sect:funct_spaces}
We define the same resolution spaces $X^s$ and $Y^s$ that were used in \cite{IP12b}.
These spaces are based on the $U^p$- and $V^p$-spaces, where we refer the reader to
\cite{HHK09,HTT11,HTT13,KTV14} for more details.

\begin{definition}\label{def:x_y}
  Let $s\in\R$.
  \begin{enumerate}
  \item We define $\widetilde X^s$ as the space of all functions $u\colon\R\to H^s(\T^3)$ such that
    $t\mapsto e^{itQ(\xi)}(\cF u(t))(\xi)$ is in $U^2(\R,\C)$ for all $\xi\in\Z^3$, endowed with the norm
\[
   \|u\|_{\widetilde X^s} \coloneqq \biggl( \sum_{\xi\in\Z^3} \la\xi\ra^{2s} \bigl\|e^{itQ(\xi)}\bigl(\cF u(t)\bigr)(\xi)
                          \bigr\|_{U^2}^2 \biggr)^{\frac12}.
\]
  \item We define $\widetilde Y^s$ as the space of all functions $u\colon\R\to H^s(\T^3)$ such that
    $t\mapsto e^{itQ(\xi)}(\cF u(t))(\xi)$ is in $V^2(\R,\C)$ for all  $\xi\in\Z^3$, endowed with the norm
\[
   \|u\|_{\widetilde X^s} \coloneqq \biggl( \sum_{\xi\in\Z^3} \la\xi\ra^{2s} \bigl\|e^{itQ(\xi)}\bigl(\cF u(t)\bigr)(\xi)
                          \bigr\|_{V^2}^2 \biggr)^{\frac12}.
\]
  \item For time intervals $I\subseteq\R$ we define $X^s(I)$ and $Y^s(I)$ to be the corresponding
    restriction spaces:
\begin{align*}
   X^s(I) &\coloneqq \Bigl\{ u\in C(I;H^s) : \|u\|_{X^s(I)}\coloneqq \sup_{\substack{J\subseteq I,\\|J|\leq1}}
    \inf_{v\cdot \cf_J(t) = u\cdot \cf_J(t)} \|v\|_{\widetilde X^s} < \infty \Bigr\},\\
   Y^s(I) &\coloneqq \Bigl\{ u\in C(I;H^s) : \|u\|_{Y^s(I)}\coloneqq \sup_{\substack{J\subseteq I,\\|J|\leq1}}
    \inf_{v\cdot \cf_J(t) = u\cdot \cf_J(t)} \|v\|_{\widetilde Y^s} < \infty \Bigr\}.
\end{align*}
  \end{enumerate}
\end{definition}

Note, that we have the following important properties.
\begin{lemma}\label{lem:lem_xy}
  Let $I\subseteq\R$ be a time interval.
  \begin{enumerate}
  \item\label{it:embed_xyl} We have $X^1(I) \hookrightarrow Y^1(I) \hookrightarrow L^\infty(I,H^1(\T^3))$.
  \item\label{it:lin_sol_x} In addition, assume that $0\in I$.
    Let $s\geq0$, $\phi\in H^s(\T^3)$, and $e^{it\Delta}\phi$ be the linear solution for $t\in I$,
    then $e^{it\Delta}\phi\in X^s(I)$ and
\[
   \|e^{it\Delta}\phi\|_{X^s(I)} \leq \|\phi\|_{H^s(\T^3)}.
\]
  \end{enumerate}
\end{lemma}

Following \cite[formula~(2.3)]{IP12b}, we introduce a critical norm $Z$ that is weaker than $X^1$,
i.e.\ $\|u\|_{Z(I)} \ls \|u\|_{X^1(I)}$, which follows essentially from Proposition~\ref{prop:lin_str_est}.
\begin{definition}\label{def:z_norm}
  Let $p_0=4+\frac{1}{10}$, $p_1=100$ and $I\subseteq\R$ be an interval, then we define the norm
\[
  \|u\|_{Z(I)} \coloneqq \sum_{p\in\{p_0,p_1\}} \sup_{J\subseteq I,\;|J|\leq1}
    \biggl( \sum_{N\geq1} N^{5-\frac{p}{2}} \|P_Nu\|_{L^p(J \times \T^3)}^p \biggr)^{\frac1p}.
\]
\end{definition}

\subsection{Definition of solutions}
\begin{definition}
  Let $I\subseteq\R$ be a time interval, and for $f\in L^1_{\mathrm{loc}}(I,L^2(\T^3))$ let $\cI_{t_0}$ be defined as
\[
   \cI_{t_0}(f)(t) \coloneqq \int_{t_0}^t e^{i(t-s)\Delta} f(s) \,ds
\]
  for $t\in I$ and $\cI_{t_0}(f)(t)\coloneqq0$ otherwise.
  We call $u\in C(I,H^1(\T^3))$ a \emph{strong solution of \eqref{eq:nls}} if $u\in X^1(I)$, and $u$ satisfies
\begin{equation*}
  u(t) = e^{i(t-t_0)\Delta}u(t_0) - i \cI_{t_0}\bigl(u|u|^4\bigr)(t)
\end{equation*}
  for all $t,t_0\in I$.
\end{definition}
 
\subsection{Dispersive estimates}
The first (scale invariant) Strichartz estimates on both $\T^3$ and $\T^3_\balpha$ have been obtained by Bourgain
\cite{B93a,B07}. Recently, Killip and Vi\c{s}an \cite[Theorem~1.1 and formula~(1.8)]{KV14} improved
Bourgain's results, building on an earlier work of Bourgain and Demeter \cite{BD14}.
\begin{prop}[Strichartz estimates, \cite{KV14}]\label{prop:lin_str_est}
  Let $p>\frac{10}{3}$, $I\subseteq\R$ be any interval with $|I|\leq 1$ and $\cC\subset\Z^3$ a cube of size
  $N\geq 1$, then
\[
   \|P_\cC e^{it\Delta} \phi\|_{L^p_{t,x}(I\times\T^3)} \ls N^{\frac32-\frac5p}\|P_\cC\phi\|_{L^2(\T^3)}
\]
  holds true for any $P_\cC\phi\in L^2(\T^3)$. In particular, for any $e^{-it\Delta}P_\cC u\in U^p(I)$ we have
\[
   \|P_\cC u\|_{L^p_{t,x}(I\times\T^3)} \ls N^{\frac32-\frac5p}\|e^{-it\Delta}P_\cC u\|_{U^p(I)}.
\]
\end{prop}

\section{The local well-posedness and stability theory}\label{sect:lpw_stability}
\noindent
Large-data local well-posedness and stability results are obtained in this section. This allows
to connect neighboring intervals of nonlinear solutions.
The results obtained here extend the existence results for general rectangular tori in \cite{S14}.
This section is almost identical to the standard torus case in \cite[Section~3]{IP12b}, hence we omit the proofs.

The following proposition extends \cite[Proposition~3.3]{IP12b} to arbitrary rectangular tori.
It is easy to see that, thanks to the known linear (Proposition~\ref{prop:lin_str_est}) and trilinear (\cite[Proposition~4.1]{S14}) estimates, the proof is identical to the one for $\T^3$
in \cite[Proposition~3.3]{IP12b}. However, we want to point out that even the linear Strichartz estimate by Bourgain
\cite[Proposition~1.1]{B07} that provides estimates of the $L^p_tL_x^4$-norm with $p>\frac{16}{3}$, suffice
to get the results in this section. That will be pursued in the forthcoming PhD thesis of the author.
\begin{prop}[Local well-posedness]\label{prop:lwp}
  Let $E\geq 1$ and $\rho\in[-1,1]$ be given.
  \begin{enumerate} 
  \item\label{it:lwp} There exists
     $\delta_0=\delta_0(E)$ such that if $\|\phi\|_{H^1(\T^3)}\leq E$ and
\[
   \|e^{i(t-t_0)\Delta}\phi\|_{Z(I)} + \|\cI_{t_0}(e)\|_{X^1(I)} \leq \delta_0
\]
    on some interval $I\ni t_0$, $|I|\leq 1$, then there exists a unique solution $u\in X^1(I)$ of
    the approximate nonlinear Schr\"odinger equation
\begin{equation}\label{eq:lwp_approx_nls}
   i\p_tu + \Delta u = \rho u|u|^4 + e
\end{equation}
  with initial data $u(t_0)=\phi$.
  Besides,
\[
   \|u(t)-e^{i(t-t_0)\Delta}\phi\|_{X^1(I)} \ls_E \|e^{i(t-t_0)\Delta}\phi\|_{Z(I)}^{\frac32} + \|\cI_{t_0}(e)\|_{X^1(I)}.
\]
    If $e=0$ and $\rho=1$, then are the quantities $E(u)$ and $M(u)$, which are defined in \eqref{eq:conserve},
    conserved on $I$.
  \item\label{it:extend} If $u\in X^1(I)$ is a solution to \eqref{eq:nls} on some open interval $I$ and
\[
   \|u\|_{Z(I)} < +\infty,
\]
    then $u$ can be extended as a nonlinear solution to a neighborhood of $\overline I$ and
\[
   \|u\|_{X^1(I)} \leq C\bigl(E(u),\|u\|_{Z(I)}\bigr).
\]
  \end{enumerate}
\end{prop}

Also, the following stability result may be obtain along the same lines as for \cite[Proposition~3.4]{IP12b}.

\begin{prop}[Stability]\label{prop:stab}
  Assume that $I$ is an open bounded interval, $\rho\in[-1,1]$, and $\widetilde u\in X^1(I)$ satisfies the
  approximate Schr\"odinger equation
\[
   i\partial_t \widetilde u + \Delta \widetilde u = \rho \widetilde u |\widetilde u|^4 + e
   \quad\text{on }I\times\T^3.
\]
  Suppose in addition that
\[
   \|\widetilde u\|_{Z(I)} + \|\widetilde u\|_{L^\infty(I,H^1(\T^3))} \leq M,
\]
  for some $M\in[1,\infty)$. Assume that $t_0\in I$ and that $\phi\in H^1(\T^3)$ is such that
  the smallness condition
\[
   \|\phi-\widetilde u(t_0)\|_{H^1(\T^3)} + \|\cI_{t_0}(e)\|_{X^1(I)} \leq \varepsilon
\]
  holds for some $0<\varepsilon<\varepsilon_1$, where $\varepsilon_1\leq 1$ is a small constant
  depending on $M$.
  Then, there exists a strong solution $u\in X^1(I)$ of the Schr\"odinger equation
\[
   i\partial_t u + \Delta u = \rho u|u|^4 \quad\text{on }I\times\T^3
\]
  such that $u(t_0)=\phi$ and
\begin{align*}
   \|u\|_{X^1(I)} + \|\widetilde u\|_{X^1(I)} &\leq C(M),\\
   \|u-\widetilde u\|_{X^1(I)} &\leq C(M)\varepsilon.
\end{align*}
\end{prop}

\section{Euclidean profiles}\label{sect:eucl_prof}
\noindent
This section is devoted to prove estimates which compare Euclidean and periodic solutions
of both linear and nonlinear Schr\"odinger equations. This kind of comparison makes sense only for short time and
in the case of rescaled initial data that concentrate at a point (see $T_N$).
In order to obtain the main result of this section (Proposition~\ref{prop:eucl_frame_sol}),
one may argue similarly as in \cite[Section~4]{IP12b} if the extinction lemma \cite[Lemma~4.3]{IP12b} is replaced
by Lemma~\ref{lem:extinct}.

Let $\eta\in C_0^\infty(\R^3)$ be a fixed spherically symmetric function with $\eta(x)=1$ for $|x|\leq 1$ and
$\eta(x)=0$ for $|x|\geq 2$.
\begin{definition}\label{def:trafos}
   For $\phi\in \dot H^1(\R^3)$ and $N\geq 1$, we define
\[
    T_N\phi\in H^1(\T^3), \quad T_N\phi(y) = N^{\frac12} \eta\bigl(N^{\frac12}\Psi^{-1}(y)\bigr)
    \phi\bigl(N \Psi^{-1}(y)\bigr),
\]
  where $\Psi\colon\{x\in\R^3 : |x|_{\infty}<\pi\}\to \T^3$, $\Psi(x)=x$.
\end{definition}
Note that the operator $T_N\colon \dot H^1(\R^3)\to H^1(\T^3)$ is linear with the property that
$\|T_N\phi\|_{H^1(\T^3)} \ls \|\phi\|_{\dot H^1(\R^3)}$.

The next lemma helps to understand the linear and as a consequence (see Proposition~\ref{prop:lwp}~\ref{it:lwp})
nonlinear solution beyond the Euclidean window.
We want to point out that an argumentation like in the following proof works for a general $3$-dimensional
manifold $\T\times M$. The difference to the proof of \cite[Lemma~4.3]{IP12b} is the weaker estimate
\eqref{eq:kernel_abs_est} and the observation that this still suffices.
\begin{lemma}[Extinction lemma]\label{lem:extinct}{\ }
    Let $\phi\in\dot H^1(\R^3)$. For any $\eps>0$, 
    there exists $T=T(\phi,\eps)$ and $N_0=N_0(\phi,\eps)$ such that for all $N\geq N_0$, it holds that
\[
   \|e^{it\Delta} (T_N\phi)\|_{Z(TN^{-2},T^{-1})} \ls \eps.
\]
\end{lemma}
\begin{proof}
  We modify the proof of \cite[Lemma~4.3]{IP12b}. For $M\geq 1$ we have that
\[
   \bigl(P_{\leq M} e^{it\Delta}(T_N\phi)\bigr)(t,x) = \frac{1}{(2\pi)^3} \int_{\T^3} K_M(t,x-y) T_N\phi(y)\, dy,
\]
  where $K_M$ is given by
\[
   K_M(t,x) = \sum_{\xi\in\Z^3} e^{i(x\cdot \xi-tQ(\xi))}\eta^3\Bigl(\frac{\xi}{M}\Bigr).
\]
  Bourgain's exponential sum estimate \cite[Lemma~3.18]{B93a} yields
\begin{equation}\label{eq:kernel_abs_est}
   |K_M(t,x)| \ls M^2 \biggl| \sum_{\xi_1\in\Z} e^{i(x_1\xi_1-t\xi_1^2)}\eta^1\Bigl(\frac{\xi_1}{M}\Bigr)^2 \biggr|
              \ls \frac{M^3}{\sqrt q\bigl(1+M\bigl|\frac{t}{2\pi}-\frac{a}{q}\bigr|^{1/2}\bigr)}
\end{equation}
  provided that
\begin{equation*}
  \frac{t}{2\pi}=\frac{a}{q}+\beta, \quad \text{$q\in\{1,\ldots,M\}$, $a\in\Z$, $(a,q)=1$,
    $|\beta|\leq(Mq)^{-1}$.}
\end{equation*}
  Dirichlet's lemma, see e.g.\ \cite[Lemma~3.31]{B93a}, and \eqref{eq:kernel_abs_est} imply for $1\leq S\leq M$,
\begin{equation}\label{eq:kernel_l_infty}
   \|K_M\|_{L^{\infty}([SM^{-2},S^{-1}]\times\T^3)} \ls S^{-\frac12}M^3.
\end{equation}
  Indeed, assume that $|t|\leq\frac1S$, and write $\frac{t}{2\pi}=\frac{a}{q}+\beta$. Since
  $|\beta|\leq\frac1M\leq\frac1S$, it follows that $\bigl|\frac{a}{q}\bigr|\leq\frac2S$. Therefore,
  either $|a|\geq 1$, which implies $q\geq\frac{S}{4}$, or $a=0$, and hence, $q=1$ because of $(a,q)=1$.
  In the first case, \eqref{eq:kernel_l_infty} follows from \eqref{eq:kernel_abs_est}:
\[
   |K_M(t,x)| \ls q^{-\frac12}M^3 \ls S^{-\frac12}M^3.
\]
  In the second case, $|\frac{t}{2\pi}-\frac{a}{q}|^{\frac12}= \frac{1}{\sqrt{2\pi}} |t|^{\frac12}$,
  and, from \eqref{eq:kernel_abs_est}, we obtain for $t\in[SM^{-2},S^{-1}]$ that
\[
   |K_M(t,x)| \ls |t|^{-\frac12}M^2 \ls S^{-\frac12}M^3.
\]

  Similarly as in \cite[Lemma~4.3]{IP12b}, we compute for $1\leq T\leq N$ and $p\in\{p_0,p_1\}$
\[
     \sum_{L\notin[NT^{-1/1000},NT^{1/1000}]} L^{(\frac5p-\frac12)p} \|e^{it\Delta}P_L(T_N\phi)\|^p_{L^p([-1,1]\times\T^3))}
     \ls_\phi T^{-\frac{p}{1000}}
\]
  and
\begin{multline*}
   \sum_{p\in\{p_0,p_1\}}
   \sum_{L\in[NT^{-1/1000},NT^{1/1000}]} L^{(\frac5p-\frac12)p} \|e^{it\Delta}P_L(T_N\phi)\|^p_{L^p([TN^{-2},T^{-1}]\times\T^3))}\\
   \ls_\phi T^{(\frac{5}{p_0}-\frac32)p_0/1000} + T^{-(\frac{5}{p_1}+\frac12)p_1/1000}.
\end{multline*}
  The result follows for $T=T(\eps,\phi)$ sufficiently large since both exponents are negative.
\end{proof}

The next proposition describes nonlinear solutions of the initial-value problem \eqref{eq:nls} with data
that concentrates at a point.

Given $f\in L^2(\T^3)$, $t_0\in\R$, and $x_0\in\T^3$, we define
\[
    (\pi_{x_0}f)(x) \coloneqq f(x-x_0) \quad\text{and}\quad
    (\Pi_{t_0,x_0}f)(x) \coloneqq (\pi_{x_0}e^{-it_0\Delta}f)(x).
\]
\begin{definition}\label{def:ren_eucl_prof}
  We define the set of \emph{renormalized Euclidean frames} as
\begin{multline*}
   \widetilde \cF_E
   \coloneqq \Bigl\{ (N_k,t_k,x_k)_{k\geq 1} : \text{$N_k\geq 1$, $N_k\to+\infty$, $t_k\to0$, $x_k\in\T^3$,}\\
      \text{and either $t_k=0$ for all $k\geq 1$ or $\lim_{k\to\infty} N_k^2|t_k|=+\infty$} \Bigr\}.
\end{multline*}
\end{definition}

\begin{prop}\label{prop:eucl_frame_sol}
  Let $\cO=(N_k,t_k,x_k)_k\in\widetilde\cF_E$ and $\phi\in\dot H^1(\R^3)$.
  \begin{enumerate}
  \item\label{it:eucl_frame_sol}
    There exist $\tau=\tau(\phi)$ and $k_0=k_0(\phi,\cO)$ such that for all $k\geq k_0$
    there is a nonlinear solution $U_k\in X^1(-\tau,\tau)$ of \eqref{eq:nls} with initial data
    $U_k(0)=\Pi_{t_k,x_k}(T_{N_k}\phi)$ and
\[
   \|U_k\|_{X^1(-\tau,\tau)} \ls_{E_{\R^3}(\phi)} 1.
\]
  \item\label{it:eucl_frame_eucl_sol}
    Let $\Delta_{\R^3}^\Theta\coloneqq \sum_{j=1}^3\theta_j\frac{\partial^2}{\partial x_j^2}$.
    Then there exists a Euclidean solution $u\in C(\R,\dot H^1(\R^3))$ of 
\[
   i\p_tu + \Delta_{\R^3}^\Theta u = u|u|^4
\]
    with scattering data $\phi^{\pm\infty}\in\dot H^1(\R^3)$ such that the following holds
    up to a subsequence: For any $\eps>0$ there exists $T_0=T(\phi,\eps)$ such that for all $T\geq T_0$ there
    is $R_0=R_0(\phi,\eps,T)$ such that for all $R\geq R_0$ there is $k_0=k_0(\phi,\eps,T,R)$ with the property
    that for any $k\geq k_0$ it holds that
\[
   \|U_k-\widetilde u_k\|_{X^1(\{|t-t_k|\leq TN_k^{-2}\}\cap\{|t|\leq T^{-1}\})} \leq \eps,
\]
    where
\[
   (\pi_{-x_k}\widetilde u_k)(t,x) = N_k^{\frac12} \eta\Bigl(\frac{N_k\Psi^{-1}(x)}{R}\Bigr)
     u\bigl(N_k^2(t-t_k),N_k\Psi^{-1}(x)\bigr).
\]
    In addition, up to a subsequence,
\[
   \|U_k(t)-\Pi_{t_k-t,x_k}T_{N_k}\phi^{\pm\infty}\|_{X^1(\{\pm(t-t_k)\geq TN_k^{-2}\}\cap\{|t|\leq T^{-1}\})} \leq \eps
\]
    for $k\geq k_0$.
  \end{enumerate}
\end{prop}

This may be proved along the lines of \cite[Proposition~4.4]{IP12b}. Here, one has to note that the
global well-posedness result in \cite[Theorem~1.1]{CKSTT08} easily extends from 
$\Delta_{\R^3}=\sum_{j=1}^3\frac{\partial^2}{\partial x_j^2}$ to $\Delta_{\R^3}^\Theta$ by scaling.

\section{Profile decomposition}\label{sect:prof_decomp}
\noindent
The essence of this section is that for every given bounded sequence of functions in $H^1(\T^3)$, we can construct
suitable profiles and express the sequence in terms of these profiles. This section is almost identical to
\cite[Section~5]{IP12b}. We sum up the main results.

\begin{definition}[Euclidean frames]\label{def:eucl_frame_prof}{\ }
  \begin{enumerate}
  \item\label{it:def_eucl_frame}
    The set of \emph{Euclidean frames} is defined as
\[
   \cF_E \coloneqq \Bigl\{ (N_k,t_k,x_k)_{k\geq 1} : \text{$N_k\geq 1$, $N_k\to+\infty$, $t_k\to0$, $x_k\in\T^3$} \Bigr\}.
\]
    We say that two frames, $(N_k,t_k,x_k)_k$ and $(N_k',t_k',x_k')_k$, are \emph{orthogonal} if
\[
    \lim_{k\to+\infty} \Bigl( \Bigl|\ln\frac{N_k}{N_k'}\Bigr| + N_k^2|t_k-t_k'| + N_k|x_k-x_k'| \Bigr) = +\infty.
\]
    Two frames that are not orthogonal are called \emph{equivalent}.
  \item\label{it:def_eucl_prof}
     If $\cO=(N_k,t_k,x_k)_k$ is a Euclidean frame and if $\psi\in\dot H^1(\R^3)$, we define the
    \emph{Euclidean profile associated to $(\psi,\cO)$} as the sequence $(\widetilde \psi_{\cO_k})_k$ in $H^1(\T^3)$:
\[
   \widetilde \psi_{\cO_k} \coloneqq \Pi_{t_k,x_k}(T_{N_k}\psi).
\]
  \item A sequence of functions $(f_k)_k\subseteq H^1(\T^3)$ is \emph{absent from a frame $\cO$}, if for
  every profile $(\widetilde\psi_{\cO_k})_k$ associated to $\cO$,
\[
   \la f_k, \widetilde\psi_{\cO_k} \ra_{H^1\times H^1(\T^3)} \to 0,\quad \text{as $k\to+\infty$.}
\]
  \end{enumerate}
\end{definition}

The profile decomposition in the next proposition is the main statement of this section. We omit the proof of
this proposition because it similar to \cite[Proposition~5.5]{IP12a}.

\begin{prop}\label{prop:prof_decomp}
  Let $(f_k)_k$ be a sequence of functions in $H^1(\T^3)$ satisfying
\[
  \limsup_{k\to+\infty} \|f_k\|_{H^1(\T^3)} \ls E
\]
  and, up to a subsequence, $f_k\wto g\in H^1(\T^3)$. Furthermore, let $I_k=(-T_k,T^k)$ be a sequence of
  intervals such that $|I_k|\to0$ as $k\to+\infty$.
  Then, there exist a sequence of pairwise orthogonal Euclidean frames $(\cO^\alpha)_\alpha$ and a
  subsequence of profiles $(\widetilde \psi^\alpha_{\cO^\alpha_k})_k$ associated to $\cO^\alpha$
  such that, after extracting a subsequence, for every $J\geq 0$,
\[
  f_k = g+\sum_{\alpha=1}^J \widetilde \psi^\alpha_{\cO^\alpha_k} + R_k^J,
\]
  where $R^J_k$ is absent from the frames $\cO^\alpha$, $1\leq \alpha\leq J$, and is small in the sense that
\[
   \limsup_{J\to+\infty} \limsup_{k\to+\infty} \sup_{N\geq1,\;t\in I_k,\;x\in\T^3} 
     N^{-\frac12} |(e^{it\Delta}P_NR^J_k)(x)| = 0.
\]
  Besides, we also have the following orthogonality relations (here $L^p=L^p(\T^3)$):
\[
  \begin{gathered}
  \begin{aligned}
     \|f_k\|_{L^2}^2 &= \|g\|_{L^2}^2 + \|R^J_k\|_{L^2}^2 + o_k(1),\\
    \|\nabla f_k\|_{L^2}^2 &= \|\nabla g\|_{L^2}^2 + \sum_{\alpha=1}^J\|\nabla_{\R^3}\psi^\alpha\|_{L^2(\R^3)}^2
        + \|\nabla R^J_k\|_{L^2}^2 + o_k(1), 
  \end{aligned}\\
     \limsup_{J\to+\infty} \limsup_{k\to+\infty} \bigg| \|f_k\|^6_{L^6} - \|g\|^6_{L^6} -
        \sum_{\alpha=1}^J\|\widetilde\psi^\alpha_{\cO^\alpha_k}\|_{L^6}^6 \bigg| = 0,
  \end{gathered}
\]
  where $o_k(1)\to 0$ as $k\to+\infty$, possibly depending on $J$.
\end{prop}

\section{Proof of the main theorem}\label{sect:proof_thm}
\noindent
In order to prove Theorem~\ref{thm:gwp}, we see from Proposition~\ref{prop:lwp}~\ref{it:extend} that it suffices
to show that solutions remain bounded in $Z$ on intervals of length at most $1$.
To prove this, we induct on the energy $E(u)$ similarly as in \cite[Section~6]{IP12b} and
\cite[Section~5]{PTW14}.

We define
\[
   \Lambda_*(L) \coloneqq \limsup_{\tau\to 0} \sup \bigl\{ \|u\|^2_{Z(I)} : E(u)\leq L,\; |I|\leq \tau \bigr\},
\]
where the supremum is taken over all strong solutions $u$ of \eqref{eq:nls} with $E(u)\leq L$ and all intervals
$I$ of length $|I|\leq \tau$.
We also define the maximal energy such that $\Lambda_*(L)$ is finite:
\[
  \Emax \coloneqq \sup\{L\in\R_+ : \Lambda_*(L)<+\infty\}.
\]
Hence, Theorem~\ref{thm:gwp} is equivalent to the following theorem.
\begin{theorem}\label{thm:emax}
  We have that $\Emax=+\infty$. In particular, every solution of \eqref{eq:nls} is global.
\end{theorem}

The proof is the same as in \cite[Theorem~6.1]{IP12b} except of replacing \cite[Lemma~7.1]{IP12b} by

\begin{lemma}\label{lem:small_e_aux1}
  Assume that $B,N\geq 2$ are dyadic numbers, and assume that $\omega\colon (-1,1)\times\T^3\to\C$ is
  a function satisfying
\[
   |\nabla^j \omega|\leq N^{j+\frac12} \cf_{\{|x|\leq N^{-1},\;|t|\leq N^{-2}\}},\quad j=0,1.
\]
  Then, for any $f\in H^1(\T^3)$,
\[
  \|\fD_{4,1}(\omega,e^{it\Delta} P_{>BN}f)\|_{L^1((-1,1),H^1(\T^3))} \ls (B^{-\frac{1}{200}}+N^{-\frac{1}{200}})\|f\|_{H^1(\T^3)}.
\]
\end{lemma}
\begin{proof}
  As one may see in the beginning of the proof of \cite[Lemma~7.1]{IP12b}, the desired result follows from 
\begin{equation}\label{eq:kernel_bd}
   \|K\|_{L^2(\T^3)\to L^2(\T^3)} \ls N^2 \bigl(B^{-\frac{1}{100}}+N^{-\frac{1}{100}}\bigr),
\end{equation}
  where $K\colon L^2(\T^3)\to L^2(\T^3)$,
\[
   K(f)(x) \coloneqq P_{>BN} \int_\R e^{-it\Delta} W(t,x)  P_{>BN} e^{it\Delta} f(x) \, dt.
\]
  For the purpose of proving \eqref{eq:kernel_bd}, we calculate the Fourier coefficients of $K$: Let $p,q\in\Z^3$,
  then
\begin{align*}
  c_{p,q} &= \la e^{ip\cdot x}, K(e^{iq\cdot x})(x) \ra_{L^2_x\times L^2_x(\T^3)}\\
         &= \int_{\R} \bigl\la P_{>BN} e^{it\Delta} e^{ip\cdot x}, W(t,x) P_{>BN} e^{it\Delta} e^{iq\cdot x}
                     \bigr\ra_{L^2_x\times L^2_x(\T^3)} \,dt\\
         &= C(1-\eta^3)\Bigl(\frac{p}{BN}\Bigr)  (1-\eta^3)\Bigl(\frac{q}{BN}\Bigr)
            (\cF_{t,x}W)\bigl(Q(p)-Q(q),p-q\bigr).
\end{align*}
  From the definition of $W$ and scaling in $t$ and $x$, we may get the estimate
\begin{equation}\label{eq:cpq_bd}
  |c_{p,q}| \ls N^{-1} \biggl(1+\frac{|Q(p)-Q(q)|}{N^2}\biggr)^{-10} \biggl(1+\frac{|p-q|}{N}\biggr)^{-10}
     \cf_{\{|p|\geq BN\}}\cf_{\{|q|\geq BN\}}.
\end{equation}
  Using Schur's lemma, we see that
\[
  \|K\|_{L^2(\T^3)\to L^2(\T^3)} \ls \sup_{p\in\Z^3}\sum_{q\in\Z^3}|c_{p,q}| + \sup_{q\in\Z^3}\sum_{p\in\Z^3}|c_{p,q}|.
\]
  In view of \eqref{eq:cpq_bd}, it suffices to prove
\begin{multline}\label{eq:kernel_bd_red}
    \sup_{|p|\geq BN} \sum_{v\in\Z^3}\biggl(1+\frac{|Q(p)-Q(p+v)|}{N^2}\biggr)^{-10}
    \biggl(1+\frac{|v|}{N}\biggr)^{-10}\\
    \ls N^3 \bigl(B^{-\frac{1}{100}} + N^{-\frac{1}{100}}\bigr)
\end{multline}
  to obtain \eqref{eq:kernel_bd}.

  Let $\theta_{\mathrm{max}}\coloneqq\max\{1,\theta_2,\theta_3\}$ and $\Theta\coloneqq\diag(1,\theta_2,\theta_3)$, then we split
  the sum over $v\in\Z^3$ into three parts:
\[
   \sum_{|v|\geq N\min(N,B)^{1/100}}
   + \sum_{\substack{|v|< N\min(N,B)^{1/100},\\|p\cdot \Theta v| \geq \theta_{\mathrm{max}} N^2\min(N,B)^{1/10}}} 
   + \sum_{\substack{|v|< N\min(N,B)^{1/100},\\|p\cdot \Theta v| < \theta_{\mathrm{max}} N^2\min(N,B)^{1/10}}}.
\]
  Thus, it suffices the show \eqref{eq:kernel_bd_red}, where we replace the sum by any sum above.
  We will call these terms $S_1$, $S_2$, and $S_3$.
  One easily verifies that $S_1\ls N^3\min(N,B)^{-1/100}$ because
\[
   S_1 \leq \sum_{|v|\geq N\min(N,B)^{\frac{1}{100}}} \biggl(1+\frac{|v|}{N}\biggr)^{-10}
   \ls  N^{10} \bigl(N\min(N,B)^{\frac{1}{100}}\bigr)^{-7}.
\]
  In order to treat $S_2$, we observe that
\[
  Q(v)\leq\theta_{\mathrm{max}}|v|^2< \theta_{\mathrm{max}} N^2\min(N,B)^{\frac{1}{50}}
  < \theta_{\mathrm{max}} N^2\min(N,B)^{\frac{1}{10}} ,
\]
  and thus,
\[
   \biggl(1+\frac{|Q(p)-Q(p+v)|}{N^2}\biggr)^{-1}
   \leq \frac{N^2}{2|p\cdot \Theta v| - Q(v)}
   \leq \frac{N^2}{|p\cdot \Theta v|}.
\]
   We may bound $S_2$ by
\begin{align*}
   N^{20} \sum_{\substack{|v|< N\min(N,B)^{1/100}\\|p\cdot \Theta v|\geq \theta_{\mathrm{max}} N^2\min(N,B)^{1/10}}}
           |p\cdot \Theta v|^{-10}
   &\leq \min(N,B)^{-1} \sum_{|v|< N\min(N,B)^{1/100}} 1 \\
   &\ls N^3 \min(N,B)^{-\frac{1}{100}}.
\end{align*}
   Finally, it remains to bound $S_3$. To that purpose, we set $\widehat p=\frac{p}{|p|}$.
   Since $|p|\geq BN$, it suffices to prove that
\begin{multline*}
   \#\bigl\{v\in\Z^3 : |v|< N\min(N,B)^{\frac{1}{100}},\;
        |\widehat p\cdot \Theta v|< \theta_{\mathrm{max}} N\min(N,B)^{-\frac{9}{10}}\bigr\}\\
   \ls N^3\min(N,B)^{-\frac{1}{100}}.
\end{multline*}
   This point-set is covered by a rectangle in $\R^3$ with two sides of length $N\min(N,B)^{\frac{1}{100}}$ and 
   one side of length $\ls_{\Theta} N\min(N,B)^{-\frac{9}{10}}$. Therefore, the point-set is bounded by
\[
    (N\min(N,B)^{\frac{1}{100}})^2 N\min(N,B)^{-\frac{9}{10}} \ls N^3\min(N,B)^{-\frac{1}{100}},
\]
  which proves \eqref{eq:kernel_bd_red}.
\end{proof}

\bibliographystyle{amsplain}

\providecommand{\twelve}[1]{2012}
\providecommand{\bysame}{\leavevmode\hbox to3em{\hrulefill}\thinspace}
\providecommand{\MR}{\relax\ifhmode\unskip\space\fi MR }
\providecommand{\MRhref}[2]{%
  \href{http://www.ams.org/mathscinet-getitem?mr=#1}{#2}
}
\providecommand{\href}[2]{#2}

\end{document}